\documentclass[12pt]{amsart}
\usepackage[cp1251]{inputenc}
\usepackage{a4wide,amsmath,verbatim,amsthm,amsfonts,amssymb,mathdots,mathrsfs,mathtools,enumitem,color}
\usepackage[hyperfootnotes=false,colorlinks=true,allcolors=blue]{hyperref}

\newcommand{\middlegamma}{\circ}
\newcommand{\genuinecharacter}{\nu}
\newcommand{\Weyl}{\mathrm}
\newcommand{\WeylCovering}{\underline}
\newcommand{\InertialSupp}{\mathrm{\Phi}}
\newcommand{\Ker}{\mathrm{Ker}}
\newcommand{\Image}{\mathrm{Im}}
\newcommand{\startGamma}{\mathrm{min}}
\newcommand{\dGamma}{d_{\Gamma}}
\newcommand{\composante}{A}
\newcommand{\countingGamma}{j}
\newcommand{\finitecyclic}{C}
\newcommand{\algebraic}{\mathbb}

\newcommand{\cntr}{Z}
\newcommand{\characF}{\operatorname{char}}

\newcommand{\id}{\operatorname{id}}
\newcommand{\Z}{\mathbb{Z}}

\newcommand{\GL}{\operatorname{GL}}

\newcommand{\rest}{\big |}

\newcommand{\wt}{\boldsymbol}
\newcommand{\InertialPairs}{\Xi(\wt{G})}

\newcommand{\abs}[1]{\left|{#1}\right|}

\theoremstyle{plain}
\newtheorem{theorem}{Theorem}

\newtheorem{lemma}[theorem]{Lemma}

\newtheorem{remark}[theorem]{Remark}

\numberwithin{theorem}{section}
\numberwithin{equation}{section}

\begin{document}
\title[Representation theory of central extensions]{A note on the representation theory of central extensions of reductive $p$-adic groups}


\author{Eyal Kaplan}
\address{Department of Mathematics, Bar Ilan University, Ramat Gan 5290002, Israel}
\email{kaplaney@gmail.com}

\author{Dani Szpruch}
\address{Department of Mathematics and Computer Science, Open University of Israel, Raanana 43107, Israel}
\email{dszpruch@openu.ac.il}

\thanks{Kaplan was supported by the ISRAEL SCIENCE FOUNDATION (grant nos. 376/21 and 421/17).}

\begin{abstract}
In this note, we verify that several fundamental results from the theory of representations of reductive $p$-adic groups,
extend to finite central extensions of these groups.
\end{abstract}
\maketitle

\section{Introduction}

In recent years there has been an effort, by several authors, to extend the Langlands Program to the realm of central extensions of
reductive groups. When considering reductive groups over non-archimedean local fields, such extensions are $l$-groups, i.e., locally compact totally disconnected groups. Thus, an abundance of
basic representation theoretic results are immediately applicable, namely the results formulated in the literature for $l$-groups. Additional well known results
are however formulated in the context of reductive groups, and those must be read with care,
because the central extensions are not even algebraic groups. Fortunately, in many cases it turns out to be straightforward to extend these results.
Moreover, in some cases very few modifications are needed. The purpose of this note is to present several of these results.

The linear versions of the results we include have been obtained over decades of research by numerous authors.
We do not provide a historical account, since the list of authors would be very long and we prefer not to risk overlooking certain contributions. Our standard references will be the works of Bernstein and Zelevinsky \cite{BZ1,BZ2} and the recent book of Renard \cite{Renard2010} who provided a very comprehensive treatment of the subject. For some background on the study of central extensions in this context see Gan \textit{et al.} \cite{GanFanWeissman2018}.

We recall the familiar flow of the results in the linear case.
Let $G$ be a totally disconnected group. If $G$ is compact, its representation theory is fairly pleasant, for example (complex, smooth) irreducible representations
are finite-dimensional and unitary, and the category of representations of $G$ is semi-simple.
The picture is far more complicated in the locally compact case. However, certain representations called finite or compact, namely representations with compactly supported matrix coefficients, still enjoy convenient properties as in the compact case. Under certain additional mild conditions on $G$, they are admissible. Clearly though, no such representations exist
unless the center $\cntr_G$ of $G$ is compact.

Let now $G$ denote the group of $F$-points of a reductive group defined over a local non-archimedean field $F$.
In order to handle the center, we can find a subgroup $G^0$ in $G$ which is normal, has a compact center and such that $\cntr_GG^0$ is of finite index in $G$.
A representation of $G$ is called supercuspidal if it restricts to a finite representation of $G^0$. In particular, supercuspidal representations are
admissible and unitary.

The properties of arbitrary irreducible representations of $G$ can now be studied by relating them, in some sense, to the supercuspidal ones.
Specifically, for a parabolic subgroup $P=M\ltimes U$ of $G$, we have the parabolic induction functor $i_{U,1}$ and the Jacquet functor $r_{U,1}$.
These functors satisfy several useful properties, e.g., they are exact and $r_{U,1}$ is left adjoint to $i_{U,1}$.
A supercuspidal representation $\rho$ can now be defined by the vanishing of $r_{U,1}(\rho)$ for all proper parabolic subgroups $P$.

In general, any irreducible representation $\pi$ can then be embedded in $i_{U,1}(\rho)$ for some supercuspidal $\rho$ and parabolic subgroup $P$. In this manner we reduce the study of $\pi$ to the study of $\rho$, using the properties of the functors $i_{U,1}$ and $r_{U,1}$. For example, using the Iwasawa decomposition of $G$, it is simple to deduce that
$i_{U,1}$ preserves admissibility, thus $\pi$ is admissible. The fact that $r_{U,1}$ preserves admissibility is also true, and the proof is based on the Iwahori factorization of small
compact neighborhoods $\Gamma$ of the identity in $G$. The fact that irreducible representations are admissible implies that representations of finite length are both admissible and finitely generated. The other direction then follows using an analysis of spaces of $\Gamma$-invariants. As a consequence we deduce that $r_{U,1}$ carries finite length representations into finite length representations. In order to show the similar property for $i_{U,1}$ one reduces to the supercuspidal case, where it is a consequence of the geometric lemma.

Consider a central extension $\wt{G}$ of $G$ by a finite abelian group $C$. Several key geometric properties are immediately inherited
from $G$. The preimage $\wt{G^0}$ of $G^0$ in $\wt{G}$ remains useful because it preserves the properties from the linear case.
Hence it is not surprising that the theory of representations of $\wt{G^0}$ plays a fundamental role here as well. In addition, because $\wt{G}$ is split (in a certain unique way) upon restriction to
any unipotent subgroup of $G$, the functors $i_{U,1}$ and $r_{U,1}$ can be applied as in the linear case.
Hecke algebras, which play an important role in the linear case, e.g., in the proof of uniform admissibility, are defined similarly, but their structure is (slightly) more delicate because the lifting of neighborhoods of the identity from $G$ to $\wt{G}$ is not canonical.

\section{Groups and central extensions}\label{Central extensions}

Let $F$ be a non-archimedean local field. Denote the characteristic of $F$ by $\characF{F}$.
Let $\algebraic{G}$ be a connected reductive algebraic group defined over $F$ and let $\finitecyclic$ be a finite abelian group.
Let $\wt{G}$ be a topological central extension of $G=\algebraic{G}(F)$ by $\finitecyclic$.
This means that there exists a short exact sequence of topological groups
\begin{align*}
1\rightarrow \finitecyclic\xrightarrow{i} \wt{G}\xrightarrow{p} G\rightarrow 1,
\end{align*}
where $i(\finitecyclic)$ is central and $p$ is a topological covering (see e.g., \cite[p.~2]{MW2}).
This sequence necessarily splits upon restriction to a sufficiently small open subgroup of $G$.
See e.g., Ban and Jantzen \cite[Lemma~2.2]{BJ} and we mention that the restriction on the characteristic of the field in
\cite{BJ} was removed in \cite{BanJantzen2016}. Thus, $\wt{G}$ is an $l$-group.

\begin{remark}
According to the classical work of Mackey \cite{Mackey1957}, the group $\wt{G}$ can be defined by a $2$-cocycle on $G$, i.e., a Borel measurable function $\sigma:G\times G\rightarrow \finitecyclic$ such that
\begin{align*}
\sigma(g,g')\sigma(gg',g'')=\sigma(g,g'g'')\sigma(g',g''),\qquad \sigma(g,e)=\sigma(e,g')=1,\qquad\forall g,g',g''\in G,
\end{align*}
where $e$ is the identity element of $G$. The elements of $\wt{G}$ can then be written as $\langle g,\epsilon\rangle$ with $g\in G$ and $\epsilon\in \finitecyclic$, and
$\langle g,\epsilon\rangle\langle g',\epsilon'\rangle=\langle gg',\epsilon\epsilon'\sigma(g,g')\rangle$. This work will not however require us to use $2$-cocycles.
\end{remark}

For any group $H$, a subgroup $H_0<H$ and $h,h'\in H$, we write ${}^hh'=hh'h^{-1}$ and
${}^hH_0=\{{}^hh_0: h_0\in H_0\}$. Since $\wt{G}$ is a central extension of $G$, the group $G$ acts on $\wt{G}$ by conjugation.
Also for any $d\in\Z$, denote $H^d=\{h^d: h\in H\}$.

The projection $p$ is a proper map, i.e., the preimage under $p$ of a compact subset is compact.
For any subset $S\subset G$ we denote by $\wt{S}=p^{-1}(S)$ its preimage under $p$.
If $H$ is a subgroup of $G$, a lifting of $H$ is a continuous embedding of groups $\varsigma_H:H\rightarrow\wt{H}$ such that $p\circ \varsigma_H=\id_H$.
Let $L<G$ be another subgroup. Note that $L$ normalizes $H$ if and only if $\wt{L}$ normalizes $\wt{H}$. Assume that this is the case.
Then, we say that $\varsigma_H$ is $L$-equivariant if ${}^l\varsigma_{H}(h)=\varsigma_H({}^{l}h)$ for all $l\in L$.

By M{\oe}glin and Waldspurger \cite[Appendix~I]{MW2} (which is also applicable locally, see \cite[p.~277]{MW2}), for any unipotent algebraic subgroup $\algebraic{U}$ of $\algebraic{G}$, $\algebraic{U}(F)$ lifts to $\wt{G}$.
This lifting is unique when $\characF{F}=0$ or $\characF{F}$ does not divide $\abs{\finitecyclic}$, where $\abs{\finitecyclic}$ is the order of $\finitecyclic$.
Moreover, let $\algebraic{P}$  be a parabolic subgroup of $\algebraic{G}$ defined over $F$ with a unipotent radical $\algebraic{U}$, and denote
$P=\algebraic{P}(F)$ and $U=\algebraic{U}(F)$. By \textit{loc. cit.} there is a unique $P$-equivariant lifting of $U$ which we denote by $\varsigma_{U}:U\mapsto\wt{G}$.
We therefore identify $U$ as a subgroup of $\wt{G}$ via $\varsigma_{U}$.
For any $g\in G$ we have $\varsigma_{{}^gU}(u)={}^{g}(\varsigma_U(u))$ for all $u\in U$.

We denote the center of a group $H$ by $\cntr_H$. For any $H<G$, we have $\cntr_{\wt H}<\wt{\cntr_H}$
($\cntr_{\wt{H}}=\cntr_{p^{-1}(H)}$, $\wt{\cntr_H}=p^{-1}(\cntr_H)$).
For any abelian subgroup $A<G$, we have $\wt{A^{\abs{\finitecyclic}}}=p^{-1}(A^{\abs{\finitecyclic}})<\cntr_{\wt{A}}$ (\cite[\S~I.1.3(2)]{MW2}).
A maximal abelian subgroup of $\wt G$ is necessarily of the form $\wt A$ for $A<G$.

Fix a maximal $F$-split torus $\algebraic{T}_0$ in $\algebraic{G}$ and a minimal parabolic subgroup $\algebraic{P}_0$ of $\algebraic{G}$, which is defined over $F$ and contains $\algebraic{T}_0$.
Denote $T_0=\algebraic{T}_0(F)$ and $P_0=\algebraic{P}_0(F)$.
Let $M_0$ be the centralizer of $T_0$ in $G$ and let $U_0$ denote the unipotent radical of $P_0$.
Thus, $P_0=M_0\ltimes U_0$ is a Levi decomposition of $P_0$.
Let $T_0^+$ (resp., $T_0^{++}$) be the dominant (resp., strictly dominant) part of $T_0$.
Let $K$ be a good maximal compact subgroup of $G$ with respect to $P_0$ (see \cite[\S~3.5]{C}, \cite[\S~V.5.1]{Renard2010}).
In particular, the Cartan decomposition $G=KM_0K$ holds.

Henceforth we say that $P$ is a parabolic subgroup of $G$ if $P=\algebraic{P}(F)$ for a parabolic subgroup $\algebraic{P}$ of $\algebraic{G}$ defined over $F$,
and the notation $P=M\ltimes U$ implies $M$ is a Levi subgroup of $P$ and $U$ is the unipotent radical of $P$. A parabolic subgroup $P$ is called standard if $P_0<P$.
Let $M^0=\bigcap_{\chi}\mathrm{ker}|\chi|$ where $\chi$ varies over the rational characters of $M$ (\cite[\S~V.2.3]{Renard2010}).
Denote by $\composante_M$ the maximal split torus in $\cntr_M$, then $\composante_M^+$ and $\composante_M^{++}$ are defined relative to $P$
(see \cite[\S~V.3, \S~V.3.21]{Renard2010}). Denote by $P^-=M\ltimes U^-$ the opposite parabolic subgroup with respect to $M$, where
$U^-$ is the unipotent subgroup opposite to $U$ (which is the unipotent radical of $P^-$).

The Weyl group $W$ of $G$, defined with respect to $T_0$, is the quotient of the normalizer of $T_0$ in $G$ by the centralizer of $T_0$ in $G$.
For each $w\in W$ we can choose a representative $\Weyl{w}\in G$, then a representative $\WeylCovering{\Weyl{w}}\in \wt{G}$ such that $p(\WeylCovering{\Weyl{w}})=\Weyl{w}$. We then have the Bruhat decomposition $\wt{G}=\coprod_{w\in W}\wt{B}\WeylCovering{\Weyl{w}}\wt{B}$, and a similar decomposition for any standard parabolic subgroups $P,Q<G$ (see \cite[Lemma~2.6]{BJ}).

We will consider a sequence of compact open subgroups $\Gamma_{\countingGamma}$, $\countingGamma\ge0$,
which forms a basis of neighborhoods of the identity in $G$. Each $\Gamma_{\countingGamma}$ is a compact open normal subgroup of $K$ and
admits an Iwahori factorization
$\Gamma_{\countingGamma}=\Gamma_{\countingGamma}^{\startGamma-}\Gamma_{\countingGamma}^{\startGamma,\middlegamma}\Gamma_{\countingGamma}^{\startGamma,+}$ where $\Gamma_{\countingGamma}^{\startGamma,-}=\Gamma_{\countingGamma}\cap U_0^-$, $\Gamma_{\countingGamma}^{\startGamma,\middlegamma}=\Gamma_{\countingGamma}\cap M_0$ and $\Gamma_{\countingGamma}^{\startGamma,+}=\Gamma_{\countingGamma}\cap U_0$. Moreover, $\Gamma_{\countingGamma}^{\startGamma,\middlegamma}$ is normalized by $\cntr_{M_0}$, and $\Gamma_{\countingGamma}^{\startGamma,+}$ (resp., $\Gamma_{\countingGamma}^{\startGamma,-}$) is normalized by $\composante_M^+$ (resp., $(\composante_M^+)^{-1}$). (See e.g., \cite[\S~V.5.2]{Renard2010}.)

Fix $\countingGamma_0$ such that $\Gamma_{\countingGamma_0}$ lifts to $\wt G$ and fix a lifting $\eta$ of $\Gamma_{\countingGamma_0}$. Note that $\eta$ is not necessarily unique.
\begin{lemma}\label{lemma:props of neighborhoods}
There exists a $\countingGamma_1>\countingGamma_0$ such that for all $\countingGamma\ge \countingGamma_1$ the following holds:
\begin{enumerate}[leftmargin=*]
\item\label{it:equivariance part 1} $\eta\rest_{\Gamma_{\countingGamma}}$ is $K$-equivariant, $\eta\rest_{\Gamma_{\countingGamma}^{\startGamma,+}}=\varsigma_{U_0}$, $\eta\rest_{\Gamma_{\countingGamma}^{\startGamma,-}}=\varsigma_{U_0^-}$ and
$\eta\rest_{\Gamma_{\countingGamma}^{\startGamma,\middlegamma}}$ is $M_0$-equivariant.
\item\label{it:equivariance part 2} More generally if $P=M\ltimes U<G$ is a standard parabolic subgroup, then $\eta\rest_{\Gamma_{\countingGamma}\cap U}=\varsigma_{U}$,
$\eta\rest_{\Gamma_{\countingGamma}\cap U^-}=\varsigma_{U^-}$ and $\eta\rest_{\Gamma_{\countingGamma}\cap M}$ is $\composante_M$-equivariant.
\end{enumerate}
\end{lemma}
\begin{proof}
Consider $k\in K$. Since $K$ normalizes $\Gamma_{\countingGamma}$, the map $x\mapsto{}^{k^{-1}}\eta({}^kx)$ is a splitting of $\Gamma_{\countingGamma}$, and because
any $2$ splittings of $\Gamma_{\countingGamma}$ give rise to a homomorphism $\Gamma_{\countingGamma}\rightarrow \finitecyclic$, there exists $\countingGamma\gg \countingGamma_0$ such that ${}^{k^{-1}}\eta({}^kx)=\eta(x)$ for all $x\in\Gamma_{\countingGamma}$. Hence ${}^{kx'}\eta(x)=\eta({}^{kx'}x)$ for all $x,x'\in\Gamma_{\countingGamma}$. Since the index of $\Gamma_{\countingGamma}$ in $K$ is finite, there is $\countingGamma\gg \countingGamma_0$ such that ${}^{k}\eta(x)=\eta({}^{k}x)$ for all $k\in K$. Thus
$\eta\rest_{\Gamma_{\countingGamma}}$ is $K$-equivariant.

The assertions $\eta\rest_{\Gamma_{\countingGamma}^{\startGamma,+}}=\varsigma_{U_0}$ and $\eta\rest_{\Gamma_{\countingGamma}^{\startGamma,-}}=\varsigma_{U_0^-}$, and in the setup of part~\eqref{it:equivariance part 2} also $\eta\rest_{\Gamma_{\countingGamma}\cap U}=\varsigma_{U}$ and $\eta\rest_{\Gamma_{\countingGamma}\cap U^-}=\varsigma_{U^-}$, are contained in the proof of \cite[Proposition~2.11]{BJ}.

We show that $\eta\rest_{\Gamma_{\countingGamma}^{\startGamma,\middlegamma}}$ is $T_0\Gamma_{\countingGamma}^{\startGamma,\middlegamma}$-equivariant.
Since $T_0$ normalizes $\Gamma_{\countingGamma}^{\startGamma,\middlegamma}$, for each $t\in T_0$, $x\mapsto{}^{t^{-1}}\eta({}^tx)$ is a splitting of $\Gamma_{\countingGamma}^{\startGamma,\middlegamma}$.
As above, because any $2$ splittings of $\Gamma_{\countingGamma}^{\startGamma,\middlegamma}$ give a homomorphism $\Gamma_{\countingGamma}^{\startGamma,\middlegamma}\rightarrow \finitecyclic$, and
since $T_0/ (T_0\cap \Gamma_{\countingGamma}^{\startGamma,\middlegamma})$ is finitely generated, we deduce that $\eta\rest_{\Gamma_{\countingGamma}^{\startGamma,\middlegamma}}$ is $T_0$-equivariant
for $\countingGamma\gg \countingGamma_0$. Hence it is also $T_0\Gamma_{\countingGamma}^{\startGamma,\middlegamma}$-equivariant. The same argument shows that,
in the setup of part~\eqref{it:equivariance part 2}, $\eta\rest_{\Gamma_{\countingGamma}\cap M}$ is $\composante_M$-equivariant.

To conclude that $\eta\rest_{\Gamma_{\countingGamma}^{\startGamma,\middlegamma}}$ is $M_0$-equivariant, first note that
$M_0/(\Gamma_{\countingGamma_0}^{\startGamma,\middlegamma}T_0)$ is a finite set, because
$M_0/M_0^0T_0$ is finite (see \cite[Proposition~V.2.6]{Renard2010} and \cite[V.3.8]{Renard2010}) and $M_0^0$ is compact (see e.g.,
\cite[p.~241]{W}).
Take a set of representatives $\{m_l\}$ for $M_0/(\Gamma_{\countingGamma_0}^{\startGamma,\middlegamma}T_0)$ which includes the identity element,
and put
$\Gamma'=\cap_l({}^{m_l}\Gamma_{\countingGamma_0}^{\startGamma,\middlegamma})$. Since $\Gamma'<\Gamma_{\countingGamma_0}^{\startGamma,\middlegamma}$, $\eta$ is a splitting of
$\Gamma'$. Let $\countingGamma\gg \countingGamma_0$ be such that $\Gamma_{\countingGamma}^{\startGamma,\middlegamma}<\Gamma'$. Then
$x\mapsto{}^{m_l^{-1}}\eta({}^{m_l}x)$ is a splitting of
$\Gamma_{\countingGamma}^{\startGamma,\middlegamma}$. As above when $\countingGamma$ is sufficiently large, ${}^{m_l}\eta(x)=\eta({}^{m_l}x)$ for all $l$ and $x\in \Gamma_{\countingGamma}^{\startGamma,\middlegamma}$. Thus $\eta\rest_{\Gamma_{\countingGamma}^{\startGamma,\middlegamma}}$ is $M_0$-equivariant.
\end{proof}
Assume that $\countingGamma\ge \countingGamma_1$.

\begin{lemma}\label{lemma:g normalizes the stabilizer Gamma l^g}
For any $g\in G$ and $x\in{}^g\Gamma_{\countingGamma}\cap\Gamma_{\countingGamma}$, we have
${}^{g^{-1}}\eta(x)=\eta({}^{g^{-1}}x)$. Moreover, ${}^g\eta(\Gamma_{\countingGamma})\cap\eta(\Gamma_{\countingGamma})=\eta({}^g\Gamma_{\countingGamma}\cap\Gamma_{\countingGamma})$.
\end{lemma}
\begin{proof}
For the first assertion, by the Cartan decomposition, it is enough to verify
it for $g\in M_0$. By the Iwahori factorization we can write $x=x^-x^{\middlegamma}x^+$ where $x^-\in{}^g\Gamma_{\countingGamma}^{\startGamma,-}\cap\Gamma_{\countingGamma}^{\startGamma,-}$,
$x^{\middlegamma}\in{}^g\Gamma_{\countingGamma}^{\startGamma,\middlegamma}\cap\Gamma_{\countingGamma}^{\startGamma,\middlegamma}$ and $x^+\in{}^g\Gamma_{\countingGamma}^{\startGamma,+}\cap\Gamma_{\countingGamma}^{\startGamma,+}$, whence it is enough to check it separately for
$x\in \Gamma_{\countingGamma}^{\startGamma,*}$ where $*\in\{-,\middlegamma,+\}$. This follows from Lemma~\ref{lemma:props of neighborhoods}.
The second assertion follows from the first.
\end{proof}

\section{Basic representation theory}\label{Basic representation theory}

We recall that \cite[Ch.~I]{BZ1} was formulated for arbitrary $l$-groups,
and is thus immediately applicable to $\wt{G}$. The bulk of \cite[Ch.~II]{BZ1} and \cite{BZ2,Renard2010} must be adapted with more care, because it was formulated for reductive groups.

Let $\mathrm{Rep}(\wt{G})$ denote the category of representations of $\wt{G}$. In this work representations are always complex and smooth.
Let $S(\wt{G})$ denote the space of Schwartz functions on $\wt{G}$,
$S^*(\wt{G})$ be the space of distributions on $S(\wt{G})$ and
$S_c^*(\wt{G})$ be the subspace of finite distributions (\cite[1.7, 1.10]{BZ1}). One can define a convolution operation on $S_c^*(\wt{G})$, turning it into an associative algebra with a unit (\cite[1.24, 1.25]{BZ1}).
Since $\finitecyclic$ is finite and central, $S(\wt{G})=\oplus_{\genuinecharacter}S_{\genuinecharacter}(\wt{G})$ where
$\genuinecharacter$ varies over the characters of $\finitecyclic$ and $S_{\genuinecharacter}(\wt{G})$ is the subspace of $(\finitecyclic,\genuinecharacter)$-equivariant functions.

\textbf{The Hecke algebras}.
The Hecke algebra of $\wt{G}$ is the convolution subalgebra $\mathcal{H}(\wt{G})\subset S_c^*(\wt{G})$ consisting of the locally constant finite distributions on $\wt{G}$. The space $S(\wt{G})$ is isomorphic to $\mathcal{H}(\wt{G})$ via $f\mapsto f \mu_{\wt{G}}$ where $\mu_{\wt{G}}$ is a Haar measure on $\wt{G}$ (\cite[1.29]{BZ1}).
Fix $\Gamma=\Gamma_{\countingGamma}$, with $\countingGamma\geq\countingGamma_1$ (see Lemma~\ref{lemma:props of neighborhoods}). Let $\chi_{\eta(\Gamma)}\in S(\wt{G})$ be the characteristic function of $\eta(\Gamma)$,
$c_{\Gamma}=(\int_{\eta(\Gamma)}d\mu_{\wt{G}}(x))^{-1}$ and $e_{\eta(\Gamma)}=c_{\Gamma}\chi_{\eta(\Gamma)}\mu_{\wt{G}}$. Denote
$\mathcal{H}_{\Gamma}=\mathcal{H}_{\Gamma}(\wt{G})=e_{\eta(\Gamma)}*\mathcal{H}(\wt{G})*e_{\eta(\Gamma)}$
(\cite[1.26, 2.10]{BZ1}), which is called the Hecke algebra of $\wt{G}$ with respect to $\eta(\Gamma)$. The functor of $\eta(\Gamma)$-invariants, i.e., the functor attaching to $\pi\in \mathrm{Rep}(\wt{G})$ the space $\pi^{\eta(\Gamma)}$ of $\eta(\Gamma)$-invariants, is exact (\cite[2.4]{BZ1}) and induces a bijection between isomorphism classes of irreducible representations of $\wt{G}$ whose space of $\eta(\Gamma)$-invariants
is nonzero and irreducible representations of $\mathcal{H}_{\Gamma}$ (\cite[2.10]{BZ1}). In addition, by the remark at the end of \cite[II.3.12]{Renard2010}
$\mathcal{H}_{\Gamma}$ is identified with the algebra of compactly supported $\eta(\Gamma)$-bi-invariant complex-valued functions on $\wt{G}$.

For $g\in\wt{G}$ set $\overline{g}=e_{\eta(\Gamma)}*e_{g}*e_{\eta(\Gamma)}\in\mathcal{H}_{\Gamma}$ ($e_{g}$ is the Dirac distribution, see \cite[1.7]{BZ1}). Also let $\chi_{\eta(\Gamma)g\eta(\Gamma)}\in S(\wt{G})$ denote the function
supported on $\eta(\Gamma)g\eta(\Gamma)$ which is left- and right-$\eta(\Gamma)$-invariant, and
$\chi_{\eta(\Gamma)g\eta(\Gamma)}(g)=1$.

As explained in \cite[II.3.12]{Renard2010}, the preimage of $\overline{g}$ under the map $S(\wt{G})\rightarrow\mathcal{H}(\wt{G})$ is a scalar multiple
of $\chi_{\eta(\Gamma)g\eta(\Gamma)}\mu_{\wt{G}}$. We compute this scalar using the compatibility results from Lemma~\ref{lemma:g normalizes the stabilizer Gamma l^g}.
\begin{lemma}\label{lemma:e_g}
$\overline{g}=c_{\Gamma}^2\dGamma(g)\chi_{\eta(\Gamma)g\eta(\Gamma)}\mu_{\wt{G}}$, where $\dGamma(g)=\int_{\eta({}^{p(g)}\Gamma\cap\Gamma)}\,d\mu_{\wt{G}}(x)$.
\end{lemma}
\begin{proof}
By the definition of $e_{\eta(\Gamma)}$ and the properties of the convolution action (\cite[II.3.10]{Renard2010}), $\overline{g}=c_{\Gamma}^2(\chi_{\eta(\Gamma)}*(e_g*\chi_{\eta(\Gamma)}))\mu_{\wt{G}}$.
Denote $f=\chi_{\eta(\Gamma)}*(e_g*\chi_{\eta(\Gamma)})\in S(\wt{G})$.
For $g_0\in\wt{G}$,
$f(g_0)=\int_{\eta(\Gamma)}\chi_{\eta(\Gamma)}(g^{-1}x^{-1}g_0)\,d\mu_{\wt{G}}(x)$.
Hence $f$ is supported in $\eta(\Gamma)g\eta(\Gamma)$.
For $g_0=x_0gx_0'$, $x_0,x_0'\in\eta(\Gamma)$, the integral equals
$\int_{\eta(\Gamma)}\chi_{\eta(\Gamma)}({}^{g^{-1}}x)\,d\mu_{\wt{G}}(x)$, which
vanishes unless $x\in{}^{g}\eta(\Gamma)\cap\eta(\Gamma)={}^{p(g)}\eta(\Gamma)\cap\eta(\Gamma)$. In the latter case by
Lemma~\ref{lemma:g normalizes the stabilizer Gamma l^g}, $x\in\eta({}^{p(g)}\Gamma\cap\Gamma)$ and the integral becomes
$\dGamma(g)$.
\end{proof}
The following result on the Hecke algebra follows from \cite[V.5.3]{Renard2010} using Lemma~\ref{lemma:props of neighborhoods} (the arguments of \textit{loc. cit.} apply verbatim). Recall that $T_0^+$ was defined in \S~\ref{Central extensions} and $\wt{T_0^+}=p^{-1}(T_0^+)$.
\begin{lemma}\label{lemma:t and t' bar multiply}
For any $k\in \wt{K}$ and $g\in\wt{G}$, $\overline{k}*\overline{g}=\overline{kg}$ and
$\overline{g}*\overline{k}=\overline{gk}$. For any $t,t'\in\wt{T_0^+}$, $\overline{t}*\overline{t'}=\overline{tt'}$.
\end{lemma}

\textbf{The contragredient}. For $\pi\in\mathrm{Rep}(\wt{G})$, denote the contragredient representation by $\pi^{\vee}$. For the basic properties of $\pi^{\vee}$ see \cite[2.13--2.14]{BZ1}.

\textbf{Induction and Jacquet functors}. The induction and the coinvariants\footnote{Called localisation in \cite{BZ2}.} functors $i$ and $r$ of \cite[1.8]{BZ2} were defined for $l$-groups, thus their properties \cite[1.9]{BZ2} remain valid
(note that the tensor product discussed in \cite[1.9g]{BZ2} referred to the setup of \cite[2.16]{BZ1} of commuting groups).

Fix a parabolic subgroup $P=M\ltimes U$ of $G$. 
Define the (normalized) induction functor $i_{\varsigma_U(U),1}$ and the (normalized) Jacquet functor $r_{\varsigma_U(U),1}$ by \cite[1.8]{BZ2} for the triplet
$(\wt{P},\wt{M},\varsigma_U(U))$. Note that 
$i_{\varsigma_U(U),1}:\mathrm{Rep}(\wt{M})\to \mathrm{Rep}(\wt{G})$ and $r_{\varsigma_U(U),1}:\mathrm{Rep}(\wt{G})\to \mathrm{Rep}(\wt{M})$. 

Denote $\Gamma^{+}=\Gamma\cap U$, $\Gamma^{\middlegamma}=\Gamma\cap M$ and $\Gamma^{-}=\Gamma\cap U^-$.
The following results hold:
\begin{itemize}[leftmargin=*]
\item The lemma of Jacquet for admissible representations: assume $P$ is a standard parabolic subgroup and $\pi\in\mathrm{Rep}(\wt{G})$ is admissible. Then, the projection
$\pi\to r_{\varsigma_U(U),1}(\pi)$ defines a surjection $\pi^{\eta(\Gamma)}\to(r_{\varsigma_U(U),1}(\pi))^{\eta(\Gamma^{\middlegamma})}$ (\cite[Th\'{e}or\`{e}me~VI.6.1]{Renard2010}).
\item Jacquet's theorem: $r_{\varsigma_U(U),1}$ preserves admissibility (\cite[Corollary~VI.6.1]{Renard2010}, see also \cite[3.14]{BZ1}).
\item For an admissible $\pi\in\mathrm{Rep}(\wt{G})$ and $t\in\wt{\composante_M^+}$ ``sufficiently small", $\pi^{\eta(\Gamma)}=\Image\pi(\overline{t})\oplus\Ker\pi(\overline{t})$, a decomposition which is independent of
$t$ (\cite[Proposition~VI.6.1]{Renard2010}).
\end{itemize}

We explain the steps needed in order to verify these results for $\wt{G}$. First note that $e_{\eta(\Gamma)}=e_{\eta(\Gamma^-)}*e_{\eta(\Gamma^{\middlegamma})}*e_{\eta(\Gamma^+)}$, where
for any compact subgroup $O<\wt{G}$, $e_O\in S^*(\wt{G})$ is essentially the distribution defined by integrating over $O$ (see \cite[\S~II.3.11]{Renard2010}).
For \cite[Th\'{e}or\`{e}me~VI.6.1]{Renard2010}: Let $t\in \wt{\composante_M^{++}}$.
By Lemma~\ref{lemma:props of neighborhoods}, ${}^{t^{-1}}\eta(\Gamma^-)={}^{t^{-1}}\varsigma_{U^-}(\Gamma^-)=\eta({}^{p(t)^{-1}}\Gamma^-)$ and also
${}^{t^{-1}}\eta(\Gamma^{\middlegamma})=\eta({}^{p(t)^{-1}}\Gamma^{\middlegamma})$. Hence $e_{\eta(\Gamma^-)}*e_t=e_t*e_{\eta(\Gamma^-)}$ and
$e_{\eta(\Gamma^{\middlegamma})}*e_t=e_t*e_{\eta(\Gamma^{\middlegamma})}$ (see \cite[1.26c]{BZ1}).

For the paragraph following the proof of \cite[Corollaire~VI.6.1]{Renard2010} note that for $t\in \wt{\composante_M^+}$, by Lemma~\ref{lemma:props of neighborhoods} $\varsigma_{U}({}^{p(t)^{-1}}\Gamma^+)={}^{t^{-1}}\eta(\Gamma^+)$. Hence
$e_{\eta(\Gamma^+)}*e_t=e_t*e_{\varsigma_{U}({}^{p(t)^{-1}}\Gamma^+)}=
e_t*e_{{}^{t^{-1}}\eta(\Gamma^+)}$
(this is used for \cite[(VI.6.1.2)]{Renard2010}) and by \cite[Lemma~2.33]{BZ1}, the kernel of the map $\pi\to\pi_{\varsigma_U(U)}$ is equal to $\bigcup_{i\in\mathbb{N}}\mathrm{Ker}(e_{\varsigma_{U}({}^{p(t)^{-i}}\Gamma^+)})
=\bigcup_{i\in\mathbb{N}}\mathrm{Ker}(e_{{}^{p(t)^{-i}}\eta(\Gamma^+)})$.

For \cite[Proposition~VI.6.1]{Renard2010} first note that since $\finitecyclic<\cntr_{\wt{G}}$, for any $\epsilon\in \finitecyclic$ and $g\in\wt{G}$ we have $\overline{\epsilon g}=e_{\epsilon}*\overline{g}$.
For $t,t'\in\wt{T_0^+}$, by Lemma~\ref{lemma:t and t' bar multiply} $\overline{t}*\overline{t'}=\overline{tt'}$ and if $tt'=\epsilon t't$ for $\epsilon\in \finitecyclic$, then $\pi(\overline{tt'})=\pi(e_{\epsilon})\pi(\overline{t't})$. Because $\pi(e_{\epsilon})$ acts by a character of $\finitecyclic$, the kernels and images of $\pi(\overline{tt'})$ and $\pi(\overline{t't})$ coincide (see \cite[2.3, 2.5]{BZ1} for the definition of $\pi$ on $S_c^*(\wt{G})$, and also \cite[\S~III.1.3--III.1.4]{Renard2010}).
This completes the verification.

With a minor abuse of notation we denote $i_{U,1}=i_{\varsigma_U(U),1}$ and $r_{U,1}=r_{\varsigma_U(U),1}$, as in the linear case (see \cite[\S~2.3]{BZ2}). 
Using \cite[1.8]{BZ2} and \cite[Th\'{e}or\`{e}me~VI.6.1]{Renard2010} we deduce that the list of properties from \cite[2.3]{BZ2} holds. Briefly,
\begin{itemize}
\item $i_{U,1}$ and $r_{U,1}$ are exact.
\item $r_{U,1}$ is left adjoint to $i_{U,1}$.
\item $i_{U,1}$ and $r_{U,1}$ are transitive with respect to inclusion of standard Levi subgroups.
\item $i_{U,1}$ commutes with taking the contragredient.
\end{itemize}

\textbf{Unitary representations}. Unitary representations are already defined in the context of $l$-groups (see \cite[\S~IV.2]{Renard2010}). Unitary admissible
representations of $\wt{G}$ are semi-simple. Admissible square-integrable modulo $\cntr_{\wt{G}}$ representations are unitary (\cite[\S~IV.3.2]{Renard2010}),
and when they are irreducible, they satisfy the Schur orthogonality relations (\cite[Lemme~IV.3.3]{Renard2010}).

\textbf{Supercuspidal representations}.
These representations are defined as in the linear case by the vanishing of the Jacquet modules (\cite[\S~VI.2.1]{Renard2010}).
A supercuspidal representation with a unitary central character is unitary.

The characterization theorem of Harish-Chandra of supercuspidal representations, according to the support of their matrix coefficients or their restriction to $G^0$, remain valid (\cite[3.20--3.24]{BZ1}, \cite[Th\'{e}or\`{e}me~VI.2.1]{Renard2010} and the following remarks), but $G^0$ is replaced by its preimage $\wt{G^0}$. To wit, first note that in analogy with the linear case, $\wt{G^0}$ is normal in $\wt{G}$, $\wt{G}/\wt{G^0}$ is abelian, $\wt{G}/(\wt{G^0}\cntr_{\wt{G}})$ is finite, $\wt{G^0}\cap \cntr_{\wt{G}}$ is compact and $\wt{K}<\wt{G^0}$ (see e.g., \cite[\S~V.2.6]{Renard2010}).
Second, we use the properties of $\Gamma^{-,\middlegamma,+}$ from Lemma~\ref{lemma:props of neighborhoods} and in particular, that $K$ normalizes $\eta(\Gamma)$, and the equivariance properties of the splittings of unipotent radicals. It also follows that $\pi\in\mathrm{Rep}(\wt{G})$ is supercuspidal if and only if $\pi^{\vee}$ is supercuspidal.

Now we can deduce \cite[Th\'{e}or\`{e}me~VI.2.2]{Renard2010} for $\wt{G}$:
\begin{theorem}\label{theorem:irreducible is admissible}
Any irreducible $\pi\in\mathrm{Rep}(\wt{G})$ is admissible.
\end{theorem}

Next, the result on the embedding of a supercuspidal constituent as a subrepresentation and the analogous statement for a quotient
(\cite[3.28--3.30]{BZ1}, \cite[\S~VI.3.6]{Renard2010}, \cite[Theorem~2.4b]{BZ2}) hold for $\wt{G}$. Note that the character $\psi$ in
\cite[3.29]{BZ1} is a character of $\wt{G}/\wt{G^0}$ and is hence trivial on $\finitecyclic$.
For \cite[3.30]{BZ1} (the proof for $\GL_n$) one uses the center of the Levi subgroup in $\wt{G}$ instead of in $G$.

\textbf{Finiteness theorems}. We start with uniform admissibility:
\begin{theorem}\label{theorem:uniform admissibility}
There exists a constant $c$ such that for any irreducible representation $\pi$ of $\wt{G}$ or $\wt{G^0}$,
the space $\pi^{\eta(\Gamma)}$ is at most $c$-dimensional ($c$ depends on $\eta(\Gamma)$).
\end{theorem}
Indeed the proof from Bernstein \cite{Bernstein1974} is applicable, because the assumptions of \cite{Bernstein1974} are satisfied with
the choices $Z=\cntr_{\wt{G}}$, $K_0=\wt{K}$, and $a_1,\ldots,a_l$ from $(\wt{T_0^+})^{|\finitecyclic|}$ (with the notation of \textit{loc. cit.}) and one uses Lemma~\ref{lemma:props of neighborhoods} to verify
assumptions $\mathrm{II.d}$ and $\mathrm{II.f}$ of \textit{loc. cit.}

Using Theorem~\ref{theorem:uniform admissibility} we obtain the following results.
First we deduce that for each $\Gamma$, there are only finitely many non-isomorphic irreducible representations of
$\wt{G^0}$ which are finite in the sense of \cite[2.40]{BZ1} (called compact in \cite[IV.1.3]{Renard2010}) and have nonzero vectors fixed by $\eta(\Gamma)$. See \cite[4.14--4.15]{BZ1} and \cite[\S~VI.3.4]{Renard2010}. For \cite[4.15]{BZ1} we replace $\cntr_G$ with $\cntr_{\wt{G}}$. Then we obtain
the results on the ``splitting off" of the supercuspidal part from an admissible representation
(\cite[Theorem~4.17]{BZ1}, \cite[\S~VI.3.4--VI.3.6]{Renard2010}, \cite[Theorem~2.4a]{BZ2}). In particular we can deduce the decomposition
statement \cite[\S~VI.3.5]{Renard2010}, but we present it below. In addition, at this point we have already established \cite[Theorem~2.4]{BZ2}, which summarizes several results on the supercuspidal part of an admissible representation.

\begin{theorem}(Howe's theorem)\label{theorem:Howe Theorem}
$\pi\in\mathrm{Rep}(\wt{G})$ is of finite length if and only if it is admissible and finitely generated.
\end{theorem}
The proof of this follows as in \cite[Theorem~4.1]{BZ1} and \cite[\S~VI.6.3]{Renard2010}, and note that for a standard parabolic subgroup $Q$ of $G$,
$\pi(\wt{G})=\pi(\wt{Q})\pi(\wt{K})$ and $K$ normalizes $\eta(\Gamma)$.

\begin{theorem}\label{theorem:Jacquet module finite len}
Let $P=M\ltimes U$ be a parabolic subgroup of $G$.
The functor $r_{U,1}$ takes finite length representations of $\wt{G}$ into finite length representations of $\wt{M}$.
\end{theorem}
See \cite[\S~VI.6.4]{Renard2010} for the proof, which is in particular based on
Theorem~\ref{theorem:Howe Theorem} and Jacquet's lemma \cite[Th\'{e}or\`{e}me~VI.6.1]{Renard2010} verified above.

The geometric lemma \cite[Theorem~5.2]{BZ2} was proved in the context of arbitrary $l$-groups. Note that the modulus characters $\varepsilon_1$ and $\varepsilon_2$ appearing in the definition of the functor $\Phi_Z$ of \cite[5.1]{BZ2} factor through $p$. The following observation is also needed for \cite[Theorem~5.2]{BZ2}. Let $H<G$ and $g\in G$. Choose $\underline{g}\in\wt{G}$ such that $p(\underline{g})=g$. Then, we have a functor $\mathrm{Rep}(\wt{H})\to\mathrm{Rep}({}^g\wt{H})$ defined by $\pi^{\underline{g}}({}^{\underline{g}}h)=\pi(h)$ ($h\in\wt{H}$). This functor is independent of the choice of $\underline{g}$. See \cite[III.2.1]{Renard2010} for a more general definition. Thus \cite[Lemma~2.12]{BZ2} which is a consequence of \cite[Theorem~5.2]{BZ2} also holds, and we mention that its proof in \cite[6.4]{BZ2} showing that in the setting of \cite[Lemma~2.12]{BZ2} the modulus character $\varepsilon$ defined using $\varepsilon_1$ and $\varepsilon_2$ is the trivial one remains valid (because $\varepsilon$ factors through $p$).

We also comment that there is no natural action of $W$ on $\mathrm{Rep}(M)$, where $M<G$ is a Levi subgroup, except when $M=T_0$. Even in that case there is no
natural action of $W$ on $\mathrm{Rep}(\wt{T_0})$, except when some assumptions are introduced, see McNamara \cite[\S~13.6, p.~310]{McNamara} (and even then the action is not natural).

Using \cite[Theorem~5.2]{BZ2} we deduce that the following results hold: \cite[Theorem~2.8]{BZ2} on the length of a representation $\pi$ parabolically induced from an irreducible
supercuspidal representation; \cite[Theorem~2.9]{BZ2} on the
Jordan--H\"{o}lder series of such representations $\pi$ with associated inducing data; and
\cite[Corollary~2.13]{BZ2} on the Jacquet modules of $\pi$.

\begin{theorem}\label{theorem:parabolic induction functor finite len}
Let $P=M\ltimes U$ be a parabolic subgroup of $G$.
The functor $i_{U,1}$ takes finite length representations of $\wt{M}$ into finite length representations of $\wt{G}$.
\end{theorem}
This result is proved as in \cite[\S~VI.6.2]{Renard2010}, which relies on \cite[Theorem~5.2]{BZ2}.

\begin{theorem}\label{theorem:generic irreducibility}(Generic irreducibility)
Let $P=M\ltimes U$ be a parabolic subgroup of $G$.
Let $\rho$ be an irreducible representation of $\wt{M}$ and $\chi$ be an unramified character of $M$.
The representation $i_{U,1}(\chi\rho)$ is irreducible for $\chi$ in a Zariski dense open set of the algebraic variety of unramified characters of $M$.
\end{theorem}
This follows from the arguments of Bernstein \textit{et al.} \cite[\S~5.4]{BernsteinDeligneKazhdan1986} and using the
Langlands quotient theorem proved for covering groups in \cite{BJ}. Note that the case where $\rho$ is supercuspidal can be proved directly (i.e., without
Langlands' result) by adapting the arguments of \cite[Th\'{e}or\`{e}me~VI.8.5]{Renard2010}. As a consequence we can also deduce, as in
Mui\'{c} \cite{Mu}, the meromorphic continuation of standard intertwining operators using Bernstein's continuation principle (\cite{Banks}).
See Li \cite[Th\'{e}or\`{e}me~2.4.1]{Li2012} (and the ensuing discussion there).

\textbf{Bernstein's decomposition}.
Let $\mathrm{Irr}(\wt{G})$ denote the set of isomorphism classes of irreducible representations of $\wt{G}$, and
$\mathrm{Irr}(\wt{G})_{\mathrm{sc}}\subset\mathrm{Irr}(\wt{G})$ denote the supercuspidal ones. Let $\mathrm{Rep}(\wt{G})_{\mathrm{sc}}$ (resp., $\mathrm{Rep}(\wt{G})_{\mathrm{ind}}$) denote the subcategory of $\mathrm{Rep}(\wt{G})$ of
representations such that all (resp., none) of their subquotients are supercuspidal.

Consider pairs $(\wt{M},\rho)$ and $(\wt{M'},\rho')$ where $M$ and $M'$ are Levi subgroups of $G$, $\rho\in\mathrm{Irr}(\wt{M})_{\mathrm{sc}}$ and $\rho'\in\mathrm{Irr}(\wt{M'})_{\mathrm{sc}}$. These pairs are called associated if there is $g\in G$ such that ${}^g\wt{M}=\wt{M'}$ and $\rho'\cong\rho^{g}$. This defines an equivalence relation and we denote by $(\wt{M},\rho)_G$ the equivalence class
of the pair $(\wt{M},\rho)$. The supercuspidal support map
attaches to each $\pi\in\mathrm{Irr}(\wt{G})$ a class $(\wt{M},\rho)_G$ such that $\pi$ is a constituent of $i_{U,1}(\rho)$ for some parabolic subgroup $P=M\ltimes U$.

Consider the coarser equivalence relation of inertial support: pairs $(\wt{M},\rho)$ and $(\wt{M'},\rho')$ are equivalent if they are associates up to an unramified twist, i.e., ${}^g\wt{M}=\wt{M'}$ and $\rho'\cong\chi\rho^{g}$ for some $g\in G$ and an unramified character $\chi$ of $M'$.
Let $\InertialPairs$ be the set of equivalence classes.
The inertial support map $\InertialSupp:\mathrm{Irr}(\wt{G})\to\InertialPairs$ is defined by composing the supercuspidal support with the inertial support class. For $[\wt{M},\rho]\in\InertialPairs$, let $\mathrm{Rep}(\wt{G})_{[\wt{M},\rho]}\subset\mathrm{Rep}(\wt{G})$ denote the subcategory of representations whose irreducible constituents belong to $\InertialSupp^{-1}([\wt{M},\rho])$.
(See \cite[\S~V.2.7, \S~VI.7.1]{Renard2010} for more details.)

\begin{theorem}\label{theorem:decompositions}
The category $\mathrm{Rep}(\wt{G})$ admits the following decompositions:
\begin{enumerate}[leftmargin=*]
\item $\mathrm{Rep}(\wt{G})=\mathrm{Rep}(\wt{G})_{\mathrm{sc}}\times \mathrm{Rep}(\wt{G})_{\mathrm{ind}}$.
\item $\mathrm{Rep}(\wt{G})_{\mathrm{sc}}=\prod_{[\wt{G},\pi]\in\InertialPairs}\mathrm{Rep}(\wt{G})_{[\wt{G},\pi]}$.
\item\label{BernsteinDecomp} (Bernstein's decomposition) $\mathrm{Rep}(\wt{G})=\prod_{[\wt{M},\rho]\in\InertialPairs}\mathrm{Rep}(\wt{G})_{[\wt{M},\rho]}$.
\end{enumerate}
\end{theorem}
The first two statements were verified above (\cite[\S~VI.3.5]{Renard2010}). We note that they are consequences of the similar decompositions for $\mathrm{Rep}(\wt{G^0})$ (\cite[\S~VI.3.4]{Renard2010}), which are based on Theorem~\ref{theorem:uniform admissibility}.
The third assertion follows from \cite[\S~VI.7.2]{Renard2010}, and we point out that the arguments in \cite[\S~VI.7.1--VI.7.2]{Renard2010} rely in particular on
\cite[Theorem~2.9]{BZ2}, \cite[Lemme~VI.3.6]{Renard2010} on embedding supercuspidal constituents and \cite[Corollary~2.13c]{BZ2}.

\textbf{Second adjointness}. We verify Bernstein's second adjointness theorem.
\begin{theorem}\label{theorem:2nd adjointness}
Let $P=M\ltimes U$ be a parabolic subgroup of $G$.
The functor $r_{U^-,1}$ on $\mathrm{Rep}(\wt{G})$ is right adjoint to the functor $i_{U,1}$ on $\mathrm{Rep}(\wt{M})$.
\end{theorem}
The theorem is equivalent to Casselman's pairing: For any $\pi\in\mathrm{Rep}(\wt{G})$,
$r_{U^-,1}(\pi^{\vee})\cong r_{U,1}(\pi)^{\vee}$. For the proof we follow \cite[\S~VI.9.6]{Renard2010} (where the equivalence was also shown).

The main ingredient is the lemma of Jacquet for non-admissible representations \cite[Th\'{e}or\`{e}me~VI.9.1]{Renard2010}. The proof occupies
\cite[\S~VI.9.1--VI.9.3]{Renard2010}. For \cite[VI.9.1]{Renard2010} we apply the same observations used for the proofs in \cite[\S~VI.6.1]{Renard2010},
e.g. Lemma~\ref{lemma:props of neighborhoods} and Lemma~\ref{lemma:t and t' bar multiply}.
For \cite[Remarques~VI.9.1 (2)--(3)]{Renard2010} we use Lemma~\ref{lemma:props of neighborhoods}.

Section \cite[\S~VI.9.2]{Renard2010} relies on the restriction of
representations from $G$ to $G^0$ (\cite[\S~VI.4.1, Proposition~VI.3.2]{Renard2010}, see \cite[3.26, 3.29]{BZ1}) which applies here to
$\wt{G}$ and $\wt{G^0}$, on \cite[Theorem~5.2]{BZ2} and \cite[Corollary~2.13]{BZ2}, Theorem~\ref{theorem:uniform admissibility},
and Theorem~\ref{theorem:generic irreducibility} (in fact \cite[Th\'{e}or\`{e}me~VI.8.5]{Renard2010} which is just the supercuspidal case).

Section \cite[\S~VI.9.3]{Renard2010} relies on Theorem~\ref{theorem:decompositions}\eqref{BernsteinDecomp}.

The remaining arguments in \cite[\S~VI.9.6]{Renard2010} are again adapted as in \cite[\S~VI.9.1]{Renard2010}.
For \cite[Proposition~VI.9.6]{Renard2010}, note that if $m\in \wt{M}$ and $t\in \wt{\composante_M}$, $e_t*e_m=e_{\epsilon}*e_m*e_t$ for some $\epsilon\in \finitecyclic$. This completes
the proof of Theorem~\ref{theorem:2nd adjointness}.

\subsection*{Acknowledgments}
We are happy to thank Dubravka Ban, Fan Gao, Chris Jantzen, Erez Lapid and David Renard for valuable discussions.
We are also grateful to the referee for numerous helpful remarks, which significantly improved the presentation.

\def\cprime{$'$} \def\cprime{$'$} \def\cprime{$'$}


\begin{thebibliography}{GGW18}

\bibitem[BJ13]{BJ}
D.~Ban and C.~Jantzen.
\newblock The {L}anglands quotient theorem for finite central extensions of
  {$p$}-adic groups.
\newblock {\em Glas. Mat. Ser. III}, 48(2):313--334, 2013.

\bibitem[BJ16]{BanJantzen2016}
D.~Ban and C.~Jantzen.
\newblock The {L}anglands quotient theorem for finite central extensions of
  {$p$}-adic groups {II}: intertwining operators and duality.
\newblock {\em Glas. Mat. Ser. III}, 51(71)(1):153--163, 2016.

\bibitem[Ban98]{Banks}
W.~D. Banks.
\newblock A corollary to {B}ernstein's theorem and {W}hittaker functionals on
  the metaplectic group.
\newblock {\em Math. Res. Lett.}, 5(6):781--790, 1998.

\bibitem[Ber84]{Bernstein1984}
J.~N. Bernstein.
\newblock Le ``centre'' de {B}ernstein.
\newblock In {\em Representations of reductive groups over a local field},
  Travaux en Cours, pages 1--32. Hermann, Paris, 1984.
\newblock Edited by P. Deligne.

\bibitem[BZ76]{BZ1}
I.~N. Bernstein and A.~V. Zelevinsky.
\newblock Representations of the group ${GL(n,F)}$ where ${F}$ is a local
  non-{A}rchimedean field.
\newblock {\em Russian Math. Surveys}, 31(3):1--68, 1976.

\bibitem[BZ77]{BZ2}
I.~N. Bernstein and A.~V. Zelevinsky.
\newblock Induced representations of reductive ${p}$-adic groups {I}.
\newblock {\em Ann. Scient. \'{E}c. Norm. Sup.}, 10(4):441--472, 1977.

\bibitem[BDK86]{BernsteinDeligneKazhdan1986}
J.~Bernstein, P.~Deligne, and D.~Kazhdan.
\newblock Trace {P}aley-{W}iener theorem for reductive {$p$}-adic groups.
\newblock {\em J. Analyse Math.}, 47:180--192, 1986.

\bibitem[Ber74]{Bernstein1974}
I.~N. Bern\v{s}te\u{\i}n.
\newblock All reductive {$p$}-adic groups are of type {I}.
\newblock {\em Funkcional. Anal. i Prilo\v{z}en.}, 8(2):3--6, 1974.

\bibitem[Car79]{C}
P.~Cartier.
\newblock Representations of ${p}$-adic groups: a survey.
\newblock In {\em Automorphic forms, representations, and ${L}$-functions},
  volume 33 Part 1, pages 111--155, 1979.

\bibitem[GGW18]{GanFanWeissman2018}
W.~T. Gan, F.~Gao, and M.~H. Weissman.
\newblock L-group and the {L}anglands program for covering groups: a historical
  introduction.
\newblock {\em Ast\'{e}risque}, 398:1--31, 2018.
\newblock L-groups and the Langlands program for covering groups.

\bibitem[Li12]{Li2012}
W.-W. Li.
\newblock La formule des traces pour les rev\^{e}tements de groupes
  r\'{e}ductifs connexes. {II}. {A}nalyse harmonique locale.
\newblock {\em Ann. Sci. \'{E}c. Norm. Sup\'{e}r. (4)}, 45(5):787--859 (2013),
  2012.

\bibitem[Mac57]{Mackey1957}
G.~W. Mackey.
\newblock Les ensembles bor\'eliens et les extensions des groupes.
\newblock {\em J. Math. Pures Appl. (9)}, 36:171--178, 1957.

\bibitem[McN12]{McNamara}
P.~J. McNamara.
\newblock Principal series representations of metaplectic groups over local
  fields.
\newblock In {\em Multiple {D}irichlet series, {L}-functions and automorphic
  forms}, volume 300 of {\em Progr. Math.}, pages 299--327.
  Birkh\"auser/Springer, New York, 2012.

\bibitem[MW95]{MW2}
C.~M{\oe}glin and J.-L. Waldspurger.
\newblock {\em Spectral decomposition and {E}isenstein series}, volume 113 of
  {\em Cambridge Tracts in Mathematics}.
\newblock Cambridge University Press, Cambridge, 1995.
\newblock Une paraphrase de l'{\'E}criture [A paraphrase of Scripture].

\bibitem[Mui08]{Mu}
G.~Mui\'{c}.
\newblock A geometric construction of intertwining operators for reductive
  ${p}$-adic groups.
\newblock {\em Manuscripta Math.}, 125:241--272, 2008.

\bibitem[Ren10]{Renard2010}
D.~Renard.
\newblock {\em Repr\'esentations des groupes r\'eductifs {$p$}-adiques},
  volume~17 of {\em Cours Sp\'ecialis\'es [Specialized Courses]}.
\newblock Soci\'et\'e Math\'ematique de France, Paris, 2010.

\bibitem[Wal03]{W}
J.-L. Waldspurger.
\newblock La formule de {P}lancherel pour les groupes {$p$}-adiques (d'apr\`es
  {H}arish-{C}handra).
\newblock {\em J. Inst. Math. Jussieu}, 2(2):235--333, 2003.

\end{thebibliography}
\end{document}